\documentclass[a4paper,noamsfonts]{amsart}
\usepackage[T1]{fontenc}
\usepackage[bitstream-charter]{mathdesign}
\usepackage{XCharter}
\usepackage{zref-clever}
\usepackage{amsthm}
\usepackage[initials]{amsrefs}
\usepackage{enumitem}
\usepackage[linktocpage=true, hidelinks, colorlinks=false]{hyperref}

\zcsetup{cap} 
\setlist{nosep} 
\zcRefTypeSetup{maintheorem}{
  Name-sg = {Theorem},
  name-sg = {theorem},
  Name-pl = {Theorems},
  name-pl = {theorems},
}

\newtheorem{maintheorem}{Theorem}
 
\newtheorem{theorem}{Theorem}[section]

\newtheorem{fact}[theorem]{Fact}

\newtheorem{lemma}[theorem]{Lemma}
\newtheorem{proposition}[theorem]{Proposition}

\theoremstyle{definition}
\newtheorem{definition}[theorem]{Definition}
\newtheorem{remark}[theorem]{Remark}

\renewcommand{\P}{\mathbb {P}}
\newcommand{\Q}{\mathbb{Q}}

\newcommand{\seb}{\subseteq}
\newcommand{\setm}{\setminus}
\newcommand{\es}{\emptyset}
\newcommand{\tleq}{\sqsubseteq}

\newcommand{\res}{\!\upharpoonright\!}
\newcommand{\tup}[1]{\langle#1\rangle}

\newcommand{\PFA}{\textsf{PFA}}
\newcommand{\CH}{\textsf{CH}}

\newcommand{\MA}{\textsf{MA}}

\newcommand{\dt}[1]{\mathring{#1}}
\newcommand{\forces}{\Vdash}
\newcommand{\ck}[1]{\check{#1}}

\newcommand{\dia}{\diamondsuit}
\newcommand{\M}{\mathbb{M}}

\newcommand{\T}{\mathcal{T}}

\newcommand{\prop}{\Xi}
\newcommand{\propp}{\prop^{+}}
\renewcommand{\S}{\mathbb{S}}

\DeclareMathOperator{\dom}{dom}
\DeclareMathOperator{\ran}{ran}
\DeclareMathOperator{\hht}{ht}
\DeclareMathOperator{\lex}{lex}

\begin{document}

\title[Countryman basis and club-isomorphisms under $\text{MA}(\omega_{1})$]
{Countryman basis and club-isomorphisms
  under $\text{MA}(\omega_{1})$}

\author{Lucas Polymeris}
\address{Department of Mathematics, Universidad de Concepción, Concepción, Chile}
\email{l.polymeris@proton.me}

\thanks{The author was partially supported by
  ANID-Subdirección de Capital Humano/Doctorado Nacional/2026-21261763}


\begin{abstract}
  We show that unlike $\PFA$,
  $\MA_{\omega_{1}}$ is consistent with the existence
  of $2^{\omega_{1}}$-many Countryman lines which pairwise contain no
  uncountable isomorphic or reverse-isomorphic suborders.
  This clarifies an issue regarding the claim
  that $\MA_{\omega_{1}}$ implies that there is a finite basis for the Countryman
  lines. The family of Countryman lines comes from
  a family of $2^{\omega_{1}}$-many coherent
  trees which pairwise contain no club-isomorphic subtrees. To prove this,
  we first construct such a family of trees using $\dia$, and then we iterate
  ccc forcings while preserving this property.
\end{abstract}

\maketitle

\section{Introduction}

A result of Shelah \cite{Shelah1976} shows that there is an
uncountable linear order $C$ such that its square $C \times C$ is the
union of countably many chain with the coordinate-wise partial order.
Such linear orders are called \emph{Countryman lines}, and are
necessarily Aronszajn, i.e., contain no copies of uncountable real
orders, $\omega_{1}$, or its reverse $\omega_{1}^{\star}$. Also, if
$C$ is any Countryman line, then $C$ and $C^{\star}$ are
\emph{far}, i.e., no uncountable linear order embeds into both.

Recall that a basis for a class of linear orders is a set
$\mathcal{B}$ such that any linear order in the class contains an
isomorphic copy of an order in $\mathcal{B}$. It follows from
the previous paragraph that any basis
for the Countryman lines, and thus for the Aronszajn lines also, must
contain at least two elements. Shelah \cite{Shelah1976} conjectured
that it is consistent that any Aronszajn line contains a Countryman
suborder. This was finally proved by Moore \cite{5-basis}; $\PFA$
implies that if $C$ is any Countryman line, then $C$ and $C^{\star}$ form
a basis for the Aronszajn lines, that is, any Aronszajn line contains
a suborder isomorphic to $C$, or to $C^{\star}$.

Observe that Moore's theorem implies that any Countryman line is a
minimal uncountable linear order, and that if $C$ and $C'$ are
Countryman lines, then they are near or reverse near, i.e,
there is an uncountable suborder of $C$ which is isomorphic or
reverse-isomorphic to a suborder of $C'$.
These two assumptions together imply that 
if $C$ is any Countryman line, then $C$ and $C^{\star}$
form a basis for the Countryman lines.
It was already known to Abraham-Shelah \cite{ShelahAbraham} that these
statements have no large cardinal strength, a fact which is still not
known for Moore's theorem.

In \cite{TodorcevicWalks}*{2.1.12}
(see \cite{univ-A-line} for a more detailed proof),
Todorcevic proved that
already $\MA_{\omega_{1}}$ implies that any Countryman line
is minimal, and at some point,
the earliest I know of is \cite{TodorcevicWalks}*{2.1.13}, it was
also claimed that $\MA_{\omega_{1}}$ implies that any two Countryman lines
are either equivalent (bi-embeddable) or reverse-equivalent.
However, it appears that several
people\footnote{In particular, this was pointed out to
  me during a talk by a mathematician which unfortunately
  I was unable to identify.}
have realized that $\MA_{\omega_{1}}$ seems to not be enough for this,
but that
$\MA_{\omega_{1}}$ + ``Any two normal Aronszajn trees are club-isomorphic''
is\footnote{The gap in the proof of \cite{TodorcevicWalks}*{2.1.13} is easily
overcome using the extra assumption.}.
The later statement has no large cardinal strength either,
and follows from $\PFA$ restricted to posets of size $\omega_{1}$.
The seemingly incorrect claim was cited by several authors and has
led to some confusion. In particular, the article of the author
\cite{PMREpi} makes this mistake\footnote{Fortunately, the only
  result of that article that suffered from this is
  Lemma 7.1, in which the hypothesis of $\MA_{\omega_{1}}$ must be complemented
  with club-isomorphisms.}, as well as \cite{univ-A-line}*{Theorem 1.1}.
It should be noted that the proofs in \cite{PMREpi,univ-A-line} are correct,
but the conclusion drawn are stronger due to the assumption of the claim.
See also \cite{MooreTodorcevic}*{Remark 4.4}.

The objective of this paper is to make this precise, and prove that in
fact $\MA_{\omega_{1}}$ is consistent with
a family of $2^{\omega_{1}}$-many pairwise far Countryman lines,
the maximum possible value. In
doing this, we will look closely to the problem of club-isomorphisms
between coherent Aronszajn trees.  For this we will apply
the non-dangerous forcing method of Abraham-Shelah
\cite{ShelahAbraham}*{Section 6}. However, there will be subtle differences
that will require to re-do most of the proofs.

Our main theorem is the following.

\begin{maintheorem}\label{thm:main}
  Consistently with $\MA$ + $2^{\omega} = \omega_{2}$ there is a family of
  $2^{\omega_{1}}$-many coherent Aronszajn trees
  which pairwise contain no club-isomorphic subtrees.
  Consequently,
  $\MA$ + $2^{\omega} = \omega_{2}$ is consistent with the existence of
  $2^{\omega_{1}}$-many Countryman lines which are pairwise far.
\end{maintheorem}

\noindent\emph{Terminology and notation:} Our notation on trees is mostly
standard. For a tree $T$, $t \in T$, and $\xi \le \hht(t)$,
we let $t\res \xi$ be the unique $t' \le_{T} t$ such that
$t' \in T_{\xi}$ (the $\xi$-th level of
$T$). For $t,s \in T$ we let $t \perp s$ mean
that $t$ and $s$ are incomparable,
$\Delta(t,s)$ denotes $\sup\{\xi \le \hht(t)\cap\hht(s)
: t\res \xi = s \res \xi\}$.
If $T$ is a tree of height $\omega_{1}$, by a \emph{subtree} of $T$,
we mean a downwards closed uncountable subset of $T$. If $X \seb \omega_{1}$,
we let $T \res X = \bigcup_{\alpha \in X}T_{\alpha}$.
Given
trees $\{T_{i} : i \in I\}$, $\bigotimes_{i \in I}T_{i}$ is the
partial order given by the coordinate-wise order on the set of $a \in
\prod_{i \in I}T_{i}$ satisfying $\hht(a(i)) = \hht(a(j))$ for all $i,j \in I$.
This is again a tree.
We will mainly work with subtrees of
$\omega^{<\omega_{1}}$, where the order is sequence extension. A subtree $T$ of
$\omega^{<\omega_{1}}$ is called \emph{coherent} if for every $t,s \in
T$ satisfying $\dom(t) \le \dom(s)$, $t =^{*} s\res\dom(t)$. Here,
$f =^{*} g$ denotes that $f$ and $g$ are functions with the same domain
that differ on finitely many values. $T$
is called \emph{finite-to-one} if every $t \in T$ is a finite-to-one
function.

Let $T$ and $S$ be $\omega_{1}$-trees. We say that $T$ \emph{club-embeds}
into $S$ if there is club $C$ and a tree embedding $f : T \res C \to S$.
By standard arguments we may assume that $f$ is height preserving,
so that in particular $\ran(f) \seb S \res C$, $f$ is called
a \emph{club-isomorphism} if $\ran(f) = S \res C$. Observe that $T$ and
$S$ have club-isomorphic subtrees if and only if some subtree of $T$
club-embeds into $S$.

For a closed $C \seb \omega_{1}$ and $\xi \in \omega_{1}$,
we let $\nu_{C}(\xi)$ be the greatest
$\nu \in C \cup \{0\}$ such that $\nu \le \xi$.

\bigskip
\noindent \textbf{Acknowledgments.}
The author thanks Justin Moore and Carlos Martínez-Ranero for various
conversations on the subject, and comments on earlier drafts of the
article. To the former also for pointing us to the club-isomorphism
problem for coherent trees. \zcref{thm:main} and other similar results
were also obtained independently by Justin Moore and Stevo Todorcevic
in \cite{MooreTodorcevic}, using rather different methods.

\section{Preliminary analysis}

Recall that our main objective is to construct models of
$\MA_{\omega_{1}}$ with large Countryman line basis. We first explain
how to reduce this problem to club-isomorphisms between coherent
trees.  We will work with a special but common type of coherent trees,
those without Suslin subtrees. Moreover we want this property to be
preserved by any $\omega_{1}$-preserving forcing extensions. For this
we adopt the following convention: we will say that $T$ is
\emph{suitable} if it is a coherent subtree of $\omega^{<\omega_{1}}$ of
height $\omega_{1}$,
which is either special or finite-to-one.
Such trees have the following easy properties which we will be using
frequently.

\begin{fact}
  Let $T$ be suitable. Then $T$ is an Aronszajn tree, every uncountable
  subset of $T$ contains an uncountable antichain, and $T$ remains
  suitable in any $\omega_{1}$-preserving forcing extension.
\end{fact}

The following lemmas make the connection between coherent trees,
embeddings between Countryman lines, and club-isomorphisms.

\begin{fact}
  Let $T$ be suitable. The natural lexicographic ordering
  of $T$, $(T,<_{\lex})$, is a Countryman line.
\end{fact}

For a proof of the finite-to-one case see
\cite{TodorcevicCoherentSeqs}*{Theorem 2.5}, and for
the special case see \cite{TodorcevicWalks}*{Theorem 4.4.1}.
Also \cite{Peng2017}*{Corollary 1} handles both at the same time.

Finally the following makes the connection with club-isomorphisms,
the proof is standard but somewhat tedious, a written proof (for the increasing
case) appears in \cite{Shalev2024}*{Lemma 5.2}, which in turns
cites \cite{PengThesis}.

\begin{fact}\label{lem:club-iso-embedding}
  Suppose that $T$ and $S$ are lexicographically ordered
  Hausdorff Aronszajn tree. If there is
  a $<_{\lex}$-increasing or $<_{\lex}$-decreasing map from
  an uncountable subset of $T$ into $S$, then some
  subtree of $T$ club-embeds into $S$.
\end{fact}

From this we conclude the following.

\begin{proposition}\label{prop:basis}
  If $\T$ is a family of suitable trees such that for all distinct
  $T,S \in \T$, $T$ and $S$ have no club-isomorphic subtrees,
  then $\{(T,<_{\lex}),(T,<_{\lex}^{*}) : T \in \T\}$
  is a family of pairwise far
  Countryman lines, and thus any basis for the Countryman lines
  has size at least $2|\T|$.
\end{proposition}

Our task then, is to construct models of $\MA_{\omega_{1}}$
where there are large families as above. For this we need two things,
to be able to construct suitable trees without club-isomorphic subtrees,
and to preserve this.

In the model of \cite{ShelahAbraham}, they iterate ccc forcings while
preserving a property on a single Aronszajn trees
$T$, which implies that no two
distinct cones of $T$ have club-isomorphic subtrees. We cannot do the
same thing and expect $T$ to be coherent;
$\MA_{\omega_{1}}$ implies that such trees are club minimal,
in the sense that they club-embed into all its subtrees \cite{Todorcevic2007}.

Fix $T$ and $S$ suitable trees. It is not hard to see that $T$ and $S$ have no
club-isomorphic subtrees iff the following holds: for every club $C$
and uncountable $A \seb T \otimes S$, there are $a,b \in A$ such that
$a(0) \perp b(0)$, $a(1) \perp b(1)$ and $\nu_{C}(\Delta(a(0),b(0)))
\neq \nu_{C}(\Delta(a(1),b(1)))$.
Imitating \cite{ShelahAbraham}, we could work with just this idea
to obtain a model of $\MA_{\omega_{1}}$ + ``there are suitable
trees $T$ and $S$ without club-isomorphic subtrees'', thus obtaining
a model of $\MA_{\omega_{1}}$ in which no basis for the Countryman lines
has less than $4$ elements. However, we will need a slightly more general
version to increase the cardinality of the basis

Let $T^{0},\dots,T^{n-1}$ be Hausdorff $\omega_{1}$ trees.
For a club $C$ and $a,b \in T^{0} \otimes \cdots \otimes T^{n-1}$, we say that
$a$ and $b$ \emph{conflict on $C$} (or are conflicting on $C$)
if for every $i < n$, $a(i) \perp b(i)$,
and for every $i < j < n$,
and $\nu_{C}(\Delta(a(i),b(i))) \neq \nu_{C}(\Delta(a(j),b(j)))$.
We will work with the following statements which generalize our previous
discussion:

\begin{itemize}[topsep=1em,align=left]
  \item[$\prop(T^{0},\dots,T^{n-1})$]
  For every club $C$, and
  for every uncountable $A \seb \otimes_{i < n}T^{i}$, there are
  $a,b \in A$ conflicting on $C$.
  \item[$\propp(T^{0},\dots,T^{n-1})$]
  For every club $C$, and
  for every uncountable $A \seb \otimes_{i < n}T^{i}$, there is an
  uncountable $B \seb A$ such that the elements of $B$ pairwise conflict
  on $C$.
\end{itemize}

\begin{definition}
  Let $\T$ be a family of $\omega_{1}$-trees. By $\prop(\T)$ we mean
  that $\prop(T_{0},\dots,T_{n-1})$ holds whenever $T_{0},\dots,T_{n-1} \in \T$
  are distinct. $\propp(\T)$ is defined analogously.
\end{definition}

If $\prop(\T)$ holds, then any two trees in $\T$ have no club-isomorphic
subtrees, and thus by \zcref{prop:basis}, if $\T$
consists of suitable trees, any basis for
the Countryman lines has size at least $2|\T|$. It is clear that
$\prop(\T)$ is formally stronger than ``the elements of $\T$ pairwise contain
no club-isomorphic subtrees''. This will be important for
proving the ccc of certain posets.

We have two tasks, to obtain models satisfying $\prop(\T)$ for a family of
suitable trees with cardinality as large as possible, and then
to iterate ccc forcings preserving $\prop$.

\section{Non-dangerous forcings}\label{sec:non-dangerous}

This section closely follows Abraham-Shelah's \cite{ShelahAbraham}*{Section 6},
so familiarity with it will make it easier to follow.

For this section fix a nonempty family $\T$ of suitable trees.
As a convention, we will let $\mathcal{A}(\T)$
denote the set of $A$ such that $A \seb \bigotimes_{i <
  n}T^{i}$ for some $1 \le n < \omega$ and distinct
$T^{0},\dots,T^{n-1}\in \T$.

\begin{definition}\label{df:safe}
  A forcing poset $\P$ is called
  \emph{$\T$-dangerous} if for some club $C$ and
  binary relation $R$ such that
  $\dom(R) \seb \P$ and $\ran(R)$ is an uncountable
  member of $\mathcal{A}(\T)$,
  whenever $(p,a),(p',a') \in R$ and $p$ is compatible
  with $p'$, $a$ does not conflict with $a'$ on $C$.
  We say that $\P$ is $\T$-safe if it is not $\T$-dangerous.
\end{definition}

Observe that $\T$-safe implies being ccc.

\begin{lemma}\label{lem:safe-preserves}
  Assume $\prop(\T)$. If $\P$ is any $\T$-safe
  forcing, then $\P \forces \prop(\T)$.
\end{lemma}
\begin{proof}
  Suppose the conclusion fails. We will show that $\P$ is not $\T$-safe.
  Since $\P \not\forces \prop(\T)$, for
  some $p \in \P$, names $\dt{A}$ and $\dt{C}$, and distinct $T_{0},\dots,
  T_{n-1} \in \T$, we have that
  $p$ forces $\dt{A} \seb T_{0} \otimes \dots \otimes T_{n-1}$
  is uncountable, $\dt{C}$ is a club,
  and that no two elements of $\dt{A}$ conflict on $\dt{C}$. Since $\P$ is
  ccc, we may assume that $\dt{C}$ is a club in the ground model,
  call it $C$ (conflicting in a coarser club implies conflicting in the
  finer).

  For each $\xi < \omega_{1}$, one finds $p_{\xi} \le p$ and
  $a_{\xi} \in T_{0}\otimes \dots \otimes T_{n-1}$ of height no less than
  $\xi$, such that $p_{\xi}\forces a_{\xi} \in \dt{A}$. Let
  $R := \{(p_{\xi},a_{\xi}) : \xi < \omega_{1}\}$. It should be clear
  that $\ran(R)$ is an uncountable subset of
  $T_{0}\otimes \dots \otimes T_{n-1}$, and that whenever $p_{\xi}$ and $p_{\eta}$
  are compatible, $a_{\xi}$ and $a_{\eta}$ do not conflict on $C$.
  So $R$ witnesses that $\P$ is not $\T$-safe.
\end{proof}

\begin{remark}\label{rmk:converse-if-ccc}
  Note that if we assume that $\P$ is ccc, then the converse also holds.
\end{remark}

The following is clear, but important.

\begin{lemma}\label{lem:psi+-safe}
  Assume $\propp(\T)$. Every ccc forcing
  is $\T$-safe.
\end{lemma}

The following says that we can iterate $\T$-safe posets. The argument
is very similar to the classical proof that the ccc is preserved under
finite support iterations. The analogous
in \cite{ShelahAbraham} is given without proof so we give a sketch.

\begin{lemma}\label{lem:safe-iterations}
  Assume $\prop(\T)$.
  Being $\T$-safe is preserved under finite support iterations.
\end{lemma}
\begin{proof}
  The proof is by induction on the length of the iteration.
  The limit case follows from the induction hypothesis exactly as in
  the proof for the preservation of the ccc. We prove the successor step.
  
  Suppose $\P$ is $\T$-safe and that $\dt{\Q}$ is a $\P$-name such that
  $\P \forces$ ``$\dt{\Q}$ is $\T$-safe''. Let
  $C$ be a club, and $R$ be such that $\dom(R) \seb \P*\dt{\Q}$
  and $\ran(R)$ is an uncountable member of $\mathcal{A}(\T)$.
  Let $\tup{((p_{\xi},\dt{q}_{\xi}),a_{\xi} ) : \xi < \omega_{1}}$
  enumerate a subset of $R$ satisfying that the $a_{\xi}$ are all distinct.
  We need to show that there are $\xi,\eta < \omega_{1}$ such that
  $(p_{\xi},\dt{q}_{\xi})$ and $(p_{\eta},\dt{q}_{\eta})$ are compatible,
  and $a_{\xi}$ and $a_{\eta}$ are conflicting on $C$.

  Let
  $\dt{\Gamma} := \{(\ck{\xi},p_{\xi}) : \xi < \omega_{1}\}$. Then
  the fact that $\P$ is ccc implies that $\P \not\forces$ ``$\dt{\Gamma}$
  is countable''. Thus for some  generic $G$ for $\P$,
  $\Gamma := \dt{\Gamma}[G]$ is
  uncountable. Since $\P \forces$ ``$\dt{\Q}$ is $\T$-safe'',
  there are $\xi \neq \eta$ in $\Gamma$
  such that $\dt{q}_{\xi}[G] \not\perp \dt{q}_{\eta}[G]$
  and $a_{\xi}$ conflicts with $a_{\eta}$ in $C$.
  Thus for some  $p' \in G$ and $\dt{q}'$,
  $p' \forces \dt{q}' \le \dt{q}_{\xi},\dt{q}_{\eta}$. Since $p_{\xi}$,
  $p_{\eta}$ and $p'$ are in $G$, we may assume that
  $p' \le p_{\xi},p_{\eta}$. Thus $(p',\dt{q}')$ witnesses that
  $(p_{\xi},\dt{q}_{\xi})$ and $(p_{\eta},\dt{q}_{\eta})$ are compatible.
\end{proof}

\begin{remark}
  Observe that in the previous proof we do not use that $\P$ is
  $\T$-safe, only that $\P$ is ccc. However its use is hidden:
  if $\P$ is not
  $\T$-safe, then $\P \not\forces \prop(\T)$ by
  \zcref{rmk:converse-if-ccc}, so
  $\P \not\forces$ ``$\dt{\Q}$ is $\T$-safe''.
\end{remark}

We now want show that given a family $\T$ of suitable trees
satisfying $\prop(\T)$, a club $C$, and $A \in \mathcal{A}(\T)$, there
is a $\T$-safe forcing introducing an uncountable $B \seb A$
consisting of elements which pairwise conflict on $C$.

First we need another fact about suitable trees
(see \cite{TodorcevicWalks} for a proof).

\begin{lemma}\label{lem:tuples}
  Suppose $T$ is coherent and every uncountable subset of $T$
  contains an uncountable antichain. Then for every $1 \le n < \omega$
  and uncountable $A \seb \bigotimes_{i < n}T$, 
  there is an uncountable $B \seb A$
  such that for all distinct $a,b \in B$, every element
  of $a$ is incomparable with every element of $b$, and for every
  $i,j < n$, $\Delta(a(i),b(i)) = \Delta(a(j),b(j))$.
\end{lemma}

\begin{theorem}\label{thm:crux}
  Assume $\prop(\T)$. For every club $C$
  and every uncountable $A \in \mathcal{A}(\T)$,
  there is a $\T$-safe forcing $\Q$ of
  size $\omega_{1}$, and family
  $\{D_{\xi} : \xi < \omega_{1}\}$ of dense subsets of $\Q$,
  such that whenever $G$ is filter of $\Q$ meeting every
  $D_{\xi}$, $\bigcup G$ is an uncountable subset of $A$ of elements which
  pairwise conflict on $C$.
\end{theorem}
\begin{proof}
  Fix $A$ and $C$ as in the statement. Then for some
  $1 \le n < \omega$, $A \seb \bigotimes_{i < n}T^{i}$ for some
  distinct $T^{0},\dots,T^{n-1} \in \T$.  By going to an uncountable
  subset of $A$ we may assume that $A$ has at most one element for any
  given height $\alpha < \omega_{1}$. We will let
  $\mathbb{Q}$ be the set of finite subsets of $A$ of elements which pairwise
  conflict on $C$ ordered by reverse inclusion.
  We want to show that $\Q$ is $\T$-safe,
  so fix a $R$ as in \zcref{df:safe} and club $E$. We may assume that
  for all $(q,a),(q',a') \in R$, if $q \neq q'$, then $a \neq a'$.
  Towards
  a contradiction suppose that for no $(q,a),(q',a') \in R$,
  $q \not\perp q'$ and $a$ conflicts with $a'$ on $E$.
  We have that
  for some $1 \le n' < \omega$ and
  distinct $S^{0},\dots,S^{n'-1} \in \T$,
  $\ran(R)$ is an uncountable subset of $\bigotimes_{i < n'}S^{i}$. The rest
  of the proof gets significantly simpler if we assume that both lists,
  $T^{0},\dots,T^{n-1}$ and $S^{0},\dots,S^{n'-1}$ are equal, so let us
  see why this reduction is possible first.

  The easier part is to see that we can assume that
  $\{T^{0},\dots,T^{n-1}\} \seb \{S^{0},\dots,S^{n'-1}\}$, if not
  simply change $R$ to $R'$ where for each $(q,a) \in R$, $(q,a') \in R'$
  is such that $a'$ is a tuple of
  $\bigotimes\{T : T \in \{T^{0},\dots,T^{n-1},S^{0},\dots,S^{n'-1}\}\}$ which
  has $a$ as a sub-tuple. Then we have that if $a'$ and $b'$ conflict on $E$,
  then so do $a$ and $b$. Because of this we can assume that we only
  have one list $T^{0},\dots,T^{n'-1}$, and that $A$ is a subset
  of $\bigotimes_{i < n}T^{i}$ for some $n \le n'$. Via a typical thinning out
  argument and $\Delta$-system lemma, we may also assume that
  for $(q,a),(q',a') \in R$, if $q \neq q'$, then $q \cap q' = \es$
  and $|q| = |q'| = k$. Now
  pick a sufficiently large regular cardinal $\theta$, and a club
  $F \seb \{N \cap \omega_{1} : |N| = \aleph_{0}, T^{0},\dots,T^{n'-1},
  A,R,C,E \in N \prec H(\theta)\}$. For each $\nu \in F$,
  pick $(q_{\nu},a_{\nu}) \in R$
  such that every element in $q_{\nu} \cup \{a_{\nu}\}$ has height
  in $\left[\nu,\min(F\setm(\nu + 1))\right[$. Each $q_{\nu}$ is a finite set of
  elements of $\bigotimes_{i<n}T^{i}$ which pairwise conflict on $C$.
  Using elementarity one easily extends each tuple to one
  in $\bigotimes_{i < n'}T^{i}$ so that the resulting tuples remain
  pairwise conflicting on $C$.

  Finally, this allows to assume that $A$ and $\ran(R)$ are both
  subsets of $\bigotimes_{i < n}T^{i}$ for some $1 \le n < \omega$ and
  distinct $T^{0},\dots,T^{n-1} \in \T$, and also the following:
  For each $(q,a),(q',a') \in R$, if $(q,a) \neq (q',a')$ then
  $q \cap q' = \es$, $|q| = |q'| =: k$, and $a \neq a'$.

  Again let $F$ be a club as defined above, and let
  $\tup{(q_{\xi},a_{\xi}) : \xi \in F}$ be such that
  for each $\xi$, every element of $q_{\xi} \cup \{a_{\xi}\}$ has height
  in $\left[\xi,\min(F\setm(\xi+1))\right[$. For each $\xi$,
  let $q_{\xi} = \{a_{\xi,0},\dots,a_{\xi,k-1}\}$ and
  $a_{\xi,k} = a_{\xi}$. We will be using the fact that
  $F \seb C \cap E$ through out the proof.

  For each $\xi \in F$ and $i < n$ let
  $a_{\xi,i}^{*} := (a_{\xi,0}(i),\dots,a_{\xi,k}(i))\res \xi$.
  Observe then that for all $j,j' \le k$,
  $\Delta(a_{\xi,j}(i),a_{\xi,j'}(i)) < \xi$ iff
  $a_{\xi,i}^{*}(j) \neq a_{\xi,i}^{*}(j')$. By pressing down
  we obtain a stationary $\Gamma \seb F$,
  $\gamma \in F$ with $\gamma < \min(\Gamma)$,
  $a^{*}_{0},\dots,a^{*}_{n-1}$ such that:
  \begin{itemize}
    \item[(a)] For every $i < n$, $a^{*}_{i}(0),\dots,a^{*}_{i}(k) \in
    T^{i}_{\gamma}$ (not necessarily distinct).
    \item[(b)] For every $i < n$, for every $\xi \in \Gamma$,
    $a^{*}_{i}  = a_{\xi,i}^{*}\res \gamma$.
    \item[(c)] For every $i < n$, for every $j,j' \le k$,
    either $a^{*}_{i}(j) \neq a^{*}_{i}(j')$, or
    $\Delta(a_{\xi,j}(i),a_{\xi,j'}(i)) \ge \xi$ for all
    $\xi \in \Gamma$.
  \end{itemize}
  Applying
  \zcref{lem:tuples} $n$-times we find uncountable
  $\Gamma' \seb \Gamma$ satisfying:
  \begin{itemize}
    \item[(d)] For all $i < n$, for all $\xi < \eta$ in $\Gamma'$,
    every element of $a^{*}_{\xi,i}$ is incomparable with every element
    of $a^{*}_{\eta,i}$ and for all $j < j' \le k$,
    $\Delta(a^{*}_{\xi,i}(j),a^{*}_{\eta,i}(j)) =
    \Delta(a^{*}_{\xi,i}(j'),a^{*}_{\eta,i}(j'))$.
    \item[(e)] For all $\xi < \eta$ in $\Gamma'$,
    every element appearing in $q_{\xi} \cup \{a_{\xi}\}$ has
    height less than $\eta$.
  \end{itemize}
  Finally an application of $\prop(\T)$ to the club $F$ and the family
  $\{(a^{*}_{\xi,0}(0),\dots,a^{*}_{\xi,n-1}(0)) : \xi \in \Gamma'\}$, gives us
  $\xi < \eta$ in $\Gamma'$ satisfying that for all $i < i' < n$,
  $\nu_{F}(\Delta(a^{*}_{\xi,i}(0),a^{*}_{\eta,i}(0))) \neq
  \nu_{F}(\Delta(a^{*}_{\xi,i'}(0),a^{*}_{\eta,i'}(0)))$. However, by
  (d) we obtain the stronger:
  \begin{itemize}
    \item[(f)] For every $i < i' < n$ and $j,j' \le k$,
    $\nu_{F}(\Delta(a^{*}_{\xi,i}(j),a^{*}_{\eta,i}(j))) \neq
    \nu_{F}(\Delta(a^{*}_{\xi,i'}(j'),a^{*}_{\eta,i'}(j')))$.
  \end{itemize}

  We claim that $q_{\xi}$ and $q_{\eta}$ are compatible, and
  that $a_{\xi}$ and $a_{\eta}$ conflict on $E$. For this we need to do
  two things:
  \begin{itemize}
    \item Show that for all $j,j' < k$, $a_{\xi,j}$ and $a_{\eta,j'}$ are
    conflicting on $C$. 
    \item Show that $a_{\xi}$ and $a_{\eta}$ conflict on $E$.
  \end{itemize}
  The second one follows directly from (a)--(f) and the fact
  that $F  \seb E$. For the first one, the same argument, using $F \seb
  C$ works if $j = j'$. If $j \neq j'$ there are a few cases. In only
  one of these cases we actually use the fact that the elements of
  $q_{\xi}$ (and $q_{\eta}$) pairwise conflict on $C$. Fix $i < i' < n$,
  we need to show that $\nu_{C}(\Delta(a_{\xi,j}(i),a_{\eta,j'}(i)))
  \neq \nu_{C}(\Delta(a_{\xi,j}(i'),a_{\eta,j'}(i')))$.

  \begin{itemize}
    \item[$\blacktriangleright$] Case
    $\xi \le \Delta(a_{\xi,j}(i),a_{\xi,j'}(i))$ and
    $\xi \le \Delta(a_{\xi,j}(i'),a_{\xi,j'}(i'))$. First observe that
    in this case we also have
    that $\eta \le \Delta(a_{\eta,j}(i),a_{\eta,j'}(i)),
    \Delta(a_{\eta,j}(i'),a_{\eta,j'}(i'))$.
    So
    $\Delta(a_{\xi,j}(i),a_{\eta,j'}(i)) = \Delta(a^{*}_{\xi,i}(j),a^{*}_{\eta,i}(j))$
    and
    $\Delta(a_{\xi,j}(i'),a_{\eta,j'}(i')) = \Delta(a^{*}_{\xi,i'}(j),a^{*}_{\eta,i'}(j))$
    so we can conclude from (f).
    \item[$\blacktriangleright$] Case
    $\Delta(a_{\xi,j}(i),a_{\xi,j'}(i)) < \xi$ and
    $\xi \le \Delta(a_{\xi,j}(i'),a_{\xi,j'}(i'))$.
    In this case
    we have that $\Delta(a_{\xi,j}(i),a_{\eta,j'}(i)) =
    \Delta(a^{*}_{i}(j),a^{*}_{i}(j')) < \gamma$
    and $\gamma \le \Delta(a_{\xi,j}(i'),a_{\eta,j'}(i')) =
    \Delta(a^{*}_{\xi,i'}(j),a^{*}_{\eta,i'}(j'))$
    hold. But $\gamma$ was chosen in $C$, so we are done.
    \item[$\blacktriangleright$] Case
    $\xi \le \Delta(a_{\xi,j}(i),a_{\xi,j'}(i))$ and
    $\Delta(a_{\xi,j}(i'),a_{\xi,j'}(i')) < \xi$.
    This is symmetric to the previous
    case.
    \item[$\blacktriangleright$]
    Case $\Delta(a_{\xi,j}(i),a_{\xi,j'}(i)) < \xi$ and
    $\Delta(a_{\xi,j}(i'),a_{\xi,j'}(i')) < \xi$. Now
    we have that
    $\alpha_{0} := \Delta(a_{\xi,j}(i),a_{\eta,j'}(i))
    = \Delta(a_{\xi,j}(i),a_{\xi,j'}(i))$
    and $\alpha_{1} := \Delta(a_{\xi,j}(i'),a_{\eta,j'}(i'))
    = \Delta(a_{\xi,j}(i'),a_{\xi,j'}(i'))$.
    But $a_{\xi,j}, a_{\xi,j'} \in q_{\xi}$,
    thus they conflict on $C$, so $\nu_{C}(\alpha_{0})
    \neq \nu_{C}(\alpha_{1})$ and we are done.
  \end{itemize}

  We now know that $\Q$ is $\T$-safe. We need to construct
  our family of $\omega_{1}$ many dense sets which will give us
  our uncountable subfamily of $A$ of pairwise conflicting elements.
  One would hope that
  for each $\gamma < \omega_{1}$,
  $\{q \in \Q : (\exists a \in q)(\hht(a) > \gamma)\}$ is dense. But this
  is not necessarily the case.

  For each $\alpha < \omega_{1}$ let
  $\Q^{\alpha} := \{q \in \Q : (\forall a \in q)(\alpha \le \hht(a))\}$.
  Because $\Q^{\alpha} \seb \Q$, it should be clear that $\Q^{\alpha}$
  is $\T$-safe. Thus to finish the proof of the theorem
  is enough to show that there is $\alpha < \omega_{1}$ such that
  for all $\gamma < \omega_{1}$,
  $\{q \in \Q^{\alpha} : (\exists a \in q)(\hht(a) > \gamma)\}$
  is dense in $\Q^{\alpha}$. To see this, assume the contrary
  and construct
  a sequence $\tup{\beta_{\xi} : \xi < \omega_{1}}$ and
  $\tup{q_{\xi} : \xi < \omega_{1}}$ such that for all $\xi < \eta$,
  \begin{itemize}
    \item $\beta_{\xi} \le \min\{\hht(a) : a \in q_{\xi}\},
    \max \{\hht(a) : a \in q_{\xi}\} < \beta_{\eta}$.
    \item For no $a \in A$ with $\beta_{\eta} \le \hht(a)$,
    $q_{\xi} \cup \{a\} \in \Q$. 
  \end{itemize}
  Then $\{q_{\xi} : \xi < \omega_{1}\}$ is an uncountable antichain
  of $\Q$, a contradiction since we proved that $\Q$ is $\T$-safe,
  and thus ccc.
\end{proof}

\section{Consistency results}

The following is the analogous of Abraham-Shelah's \cite{ShelahAbraham}*{6.15}.

\begin{theorem}\label{thm:ma+psi}
  Let $\T$ be a family of suitable trees. Assume
  $\prop(\T)$ and $2^{\omega_{1}} = \omega_{2}$.
  Then
  there is a ccc forcing $\M(\T)$ of size $\omega_{2}$
  such that
  $\MA$ + $\propp(\T)$ + $2^{\omega} = 2^{\omega_{1}} = \omega_{2}$
  holds in any extension by
  $\M(\T)$.
\end{theorem}
\begin{proof}
  This is just like the proof of $\MA + 2^{\omega} = 2^{\omega_{1}} = \omega_{2}$,
  but we just force with $\T$-safe forcing posets. Then, by
  \zcref{lem:safe-iterations, lem:safe-preserves}
  one obtains a model
  of $\MA_{\omega_{1}}$ restricted to safe posets. By \zcref{thm:crux},
  $\propp(\T)$ holds in this model,
  and thus \zcref{lem:psi+-safe} implies that every
  ccc forcing is actually $\T$-safe, therefore the full
  $\MA_{\omega_{1}}$ holds.
\end{proof}

Now, this implies that
to finish the proof of \zcref{thm:main} it is enough to
show the consistency of ``$2^{\omega_{1}} = \omega_{2}$
+ there is a family $\T$ of $2^{\omega_{1}}$-many suitable trees
satisfying $\prop(\T)$''. The rest of this section is devoted to show that
the existence of such a family follows from
$\dia$, which of course is consistent with $2^{\omega_{1}} = \omega_{2}$.
We will adapt the construction
of $2^{\omega_{1}}$-many pairwise non club-isomorphic subtrees from
\cite{ShelahAbraham} and \cite{MartinezThesis}. A similar construction
is also done in \cite{ChavezKrueger}. 

Let us denote by $\S$ the subtree of $\omega^{<\omega_{1}}$ of finite-to-one
sequences.
For $\alpha \le \omega_{1}$, by a \emph{uniformly
  coherent subtree of $\S\res \alpha$},
we mean a coherent subset $T \seb \S\res\alpha$ that is closed under
taking initial segments and finite modifications,
and has height $\alpha$. We say that $b : \alpha \to \omega$
is \emph{good for $T$}, if $b \in \S_{\alpha}$ and
for all $\xi < \alpha$, $b\res \xi \in T$. Then
is easy to see
that there is a unique way of extending $T$ to
a uniformly coherent subtree of $\S \res (\alpha+1)$ that contains
$b$. We call this the \emph{extension of $T$ given by adjoining $b$}.

For $n \in \omega$ and $x,y \in \omega^{<\omega_{1}}$.
we let $x \tleq_{n} y$ denote that $x \tleq y$ and for all $\xi \in
\dom(y)\setm\dom(x)$, $y(\xi) \ge n$.
The following is easy to verify, and will be our main tool to extend
uniformly coherent trees.

\begin{lemma}\label{lem:fund-ext}
  Let $\alpha$ be a countable limit ordinal. Let $T$ be a uniformly
  coherent subtree of $\S \res \alpha$. Suppose that
  $\{t_{n} : n < \omega\} \seb T$ is such that for all $n < \omega$,
  $t_{n} \tleq_{n} t_{n+1}$ and
  $\sup\{\dom(t_{n}) : n < \omega\} = \alpha$. Then
  $b := \bigcup_{n < \omega}t_{n}$ is good for $T$.
\end{lemma}

By a \emph{finite modifier} we mean a partial finite function
$m : \omega_{1} \to \omega$. Given $t \in \S$ and a finite modifier
$m$, $mt : \dom(t) \to \omega$ is defined by letting
$mt(\xi) = t(\xi)$ if $\xi \notin \dom(m)$, and
$mt(\xi) = m(\xi)$ otherwise.

\begin{theorem}\label{lem:hard-ext}
  Let $\alpha$ be a countable limit ordinal,
  $n$ be a positive integer, $T^{0},\dots,T^{n-1}$ be uniformly coherent
  subtrees of $\S \res \alpha$, $C \seb \alpha$ club,
  and $A \seb \bigotimes_{i < n}T^{i}$ be maximal with respect
  to having no two elements conflicting on $C$. Then there are
  $b_{0},\dots,b_{n-1} \in \S_{\alpha}$ such that
  for all $i < n$, $b_{i}$ is good for $T^{i}$,
  and whenever
  $b^{'}_{0} =^{*} b_{0}, \cdots, b^{'}_{n-1} =^{*} b_{n-1}$, some member
  of $A$ conflicts on $C$ with
  $(b_{0}^{'},\dots,b_{n-1}^{'})$.
\end{theorem}
\begin{proof}
  First note that we may assume that $\min(C) = 0$.
  Now fix
  $c \seb C$ cofinal in $\alpha$, of order type $\omega$
  satisfying that $0 \in c$, and for any two consecutive elements of $c$,
  there are at least $n-1$ element of $C$ strictly between them.
  For $\xi \in c$
  we let $\xi^{+}$ denote its successor in $c$, i.e., $\min(c \setm (\xi +1))$.

  Let $\tup{(m_{0,\xi},\dots,m_{n-1,\xi}) : \xi \in c}$ enumerate all
  $n$-element lists
  of finite modifiers with domain contained in $\alpha$,
  allow repetitions so that each element
  appears infinitely many times. We may also assume that
  $\dom(m_{i,\xi}) \seb \xi$ for all $\xi \in c$ and $i < n$.

  For each $i < n$, we will let $b_{i} := \bigcup_{\xi \in c}b_{i}(\xi)$,
  where for all $\xi \in c$, $b_{i}(\xi) \in T^{i}_{\xi}$ and
  $b_{i}(\xi) \tleq_{|c \cap \xi|} b_{i}(\xi^{+})$. By \zcref{lem:fund-ext},
  this will imply that each $b_{i}$ will indeed be good for $T^{i}$.
  We construct
  $b_{i}(\xi)$ simultaneously for all $i$ by recursion on $\xi$.
  Thus we must let $b_{i}(0) = \es$ for all $i < n$. Suppose now that
  for some $\xi \in c$ we have already defined
  $b_{i}(\xi)$ for all $i < n$. Let $k := |c \cap \xi|$. Now we explain
  how to define the $b_{i}(\xi^{+})$.

  We claim that
  there exists
  $(x_{0},\dots,x_{n-1}) \in \bigotimes_{i <n}T^{i}_{\xi^{+}}$ such that
  for all $i <n$, $b_{i}(\xi) \tleq_{k} x_{i}$, and
  $(m_{0,\xi}x_{0},\dots,m_{n-1,\xi}x_{n-1}) \notin A$.
  Suppose not. By assumption on $c$,
  there are $\xi = \eta_{0} < \dots < \eta_{n} = \xi^{+}$ all in
  $C$.
  Using the uniform coherence we easily find
  $(x_{0},\dots,x_{n-1}),\allowbreak (y_{0},\dots,y_{n-1})
  \in \bigotimes_{i <n}T^{i}_{\xi^{+}}$ such that for all $i < n$,
  $b_{i}(\xi) \tleq_{k} x_{i},y_{i}$,
  $y_{i} \neq x_{i}$, and
  $\eta_{i} \le \Delta(x_{i},y_{i}) < \eta_{i+1}$.
  Since for all $i < n$, $\dom(m_{i,\xi}) \seb \xi$, we
  have that $m_{i,\xi}y_{i} \neq m_{i,\xi}x_{i}$ and
  $\eta_{i} \le \Delta(m_{i,\xi}x_{i},m_{i,\xi}y_{i}) < \eta_{i+1}$.
  By assumption, both
  $(m_{0,\xi}x_{0},\dots,m_{n-1,\xi}x_{n-1})$ and $(m_{0,\xi}y_{0},\dots,m_{n-1,\xi}y_{n-1})$
  are elements of $A$. But this is impossible since by construction
  $(m_{0,\xi}x_{0},\dots,m_{n-1,\xi}x_{n-1})$ and $(m_{0,\xi}y_{0},\dots,m_{n-1,\xi}y_{n-1})$
  are conflicting on $C$.
  
  Using the claim we define
  $(b_{0}(\xi^{+}),\dots,b_{n-1}(\xi^{+}))$
  as any $(x_{0},\dots,x_{n-1}) \in \bigotimes_{i <n}T^{i}_{\xi^{+}}$ such that
  for all $i <n$, $b_{i}(\xi) \tleq_{k} x_{i}$, and
  $(m_{0,\xi} x_{0},\dots,m_{n-1,\xi} x_{n-1}) \notin A$. This finishes the
  construction.
  
  Let $b_{0}',\dots,b_{n-1}'$ be
  as in the statement of the lemma. Then there are finite modifiers
  $m_{0},\dots,m_{n-1}$, such that for all $i < n$, $\dom(m_{i}) \seb \alpha$
  and $b_{i}' = m_{i} b_{i}$. We need to show
  that $(m_{0}b_{0},\dots,m_{n-1}b_{n-1})$ conflicts on $C$
  with an element of $A$.
  Let $\xi \in c$ be such that
  $(m_{0},\dots,m_{n-1}) = (m_{0,\xi},\dots,m_{n-1,\xi})$. Then, by our construction,
  $(b_{0}(\xi^{+}),\dots,b_{n-1}(\xi^{+}))$ was chosen so that
  $(m_{0,\xi}b_{0}(\xi^{+}),\dots,\allowbreak m_{n-1,\xi}b_{n-1}(\xi^{+}))$ is
  not in $A$. Then, the maximality of $A$ implies that there must be
  some $a \in A$ conflicting on $C$ with
  $(m_{0,\xi}b_{0}(\xi^{+}),\dots,\allowbreak m_{n-1,\xi}b_{n-1}(\xi^{+}))$. Clearly
  $(b_{0}',\dots,b_{n-1}') = (m_{0}b_{0},\dots,m_{n-1}b_{n-1})$ extends
  $(m_{0,\xi}b_{0}(\xi^{+}),\dots,\allowbreak m_{n-1,\xi}b_{n-1}(\xi^{+}))$
  so we conclude that
  $(b_{0}',\dots,b_{n-1}')$ also conflicts on $C$ with $a$.
\end{proof}

Via a typical coding and decoding argument we obtain the following
variation of $\dia$. (see \cite{BrodskyRinotI} for a more general version).

\begin{lemma}\label{lem:diamond}
  Assume $\dia$. There
  is a sequence $\tup{X_{\alpha} : \alpha < \omega_{1}}$ of
  elements of $H(\omega_{1})$ such that
  whenever $1 \le n < \omega$, $t_{0},\dots,t_{n-1} \in 2^{\omega_{1}}$,
  $C \seb \omega_{1}$, and
  $A \seb \bigotimes_{i < n}\S$ is such that for all $\alpha < \omega_{1}$,
  $A\res \alpha := \{a \in A : \hht(a) < \alpha\}$ is countable,
  there are stationarily many $\alpha < \omega_{1}$ such that
  \begin{itemize}
    \item $X_{\alpha} = (t_{\alpha,0},\dots,t_{\alpha,n-1},c_{\alpha},A_{\alpha})
    \in (2^{\alpha})^{n}\times [\alpha]^{\le\omega}\times
    [\bigotimes_{i <n} \S\res \alpha]^{\le \omega}$.
    \item $C \cap \alpha = c_{\alpha}$.
    \item For all $i < n$, $t_{i}\res \alpha = t_{\alpha,i}$.
    \item $A \res \alpha = A_{\alpha}$.
  \end{itemize}
  Moreover, replacing some of the $X_{\alpha}$ with dummy values
  we  may assume that all of the $X_{\alpha}$ are
  elements of $(2^{\alpha})^{n}\times [\alpha]^{\le\omega}\times
    [\bigotimes_{i <n} \S\res \alpha]^{\le\omega}$
  for some $1 \le n < \omega$.
\end{lemma}

\begin{theorem}\label{thm:existence}
  Assume $\dia$. There is a family $\T$ of uniformly coherent
  subtrees of $\S$ satisfying $\prop(\T)$ and $|\T| = 2^{\omega_{1}}$.
\end{theorem}
\begin{proof}
  We build a sequence $\tup{T^{s} : s \in 2^{\le\omega_{1}}}$
  such that for all $\alpha \le \omega_{1}$ and $s \in 2^{\alpha}$,
  $T^{s}$ is a uniformly coherent subtree of $\S \res \alpha$,
  and for all $\beta < \alpha$, $T^{s\res\beta} = T^{s}\res\beta$.
  We want $\{T^{t} : t \in 2^{\omega_{1}}\}$ to witness the theorem.
  Fix $\tup{X_{\alpha} : \alpha < \omega_{1}}$
  a $\dia$-sequence as in \zcref{lem:diamond}.

  Recursively for $\alpha < \omega_{1}$, we define
  $T^{t}$ for all $t \in 2^{\alpha}$. Observe that we are forced
  to set $T^{\es} = \es$ and $T^{\tup{0}} = T^{\tup{1}} = \{\es\}$. This handles
  $\alpha = 0$ and $\alpha = 1$. At
  limits we must take the union, i.e., if $\alpha$ is limit, then
  $T^{t} := \bigcup_{\xi < \alpha}T^{t\res \xi}$ for all $t \in 2^{\alpha}$.
  Now assume that for some $0 < \alpha$
  we have defined $T^{t}$ for all $t \in 2^{\alpha}$, we need to define
  $T^{t^{\frown}\tup{0}}$ and $T^{t^{\frown}\tup{1}}$ for all $t \in 2^{\alpha}$.
  If $\alpha$ is successor, then for every $t \in 2^{\alpha}$,
  $T^{t}$ has a last level, and thus the uniform coherence
  says that both, $T^{t^{\frown}\tup{0}}$ and $T^{t^{\frown}\tup{1}}$,
  must be equal to $T^{t} \cup
  \{x^{\frown}\tup{n} : x \in T^{t}_{\alpha-1}, n < \omega\}$.
  It only remains to define
  $T^{t^{\frown}\tup{0}}$ and $T^{t^{\frown}\tup{1}}$
  for $t \in 2^{\alpha}$ with limit $\alpha$.

  Let $1 \le n < \omega$ be such that $X_{\alpha} =
  (t_{\alpha,0},\dots,t_{\alpha,n-1},c_{\alpha},A_{\alpha}) \in
  (2^{\alpha})^{n}\times [\alpha]^{\le\omega}\times
    [\bigotimes_{i <n} \S\res \alpha]^{\le\omega}$.
  For each $t \in 2^{\alpha}$, fix  ${T'}^{t}$
  some uniformly coherent
  subtree of $\S \res (\alpha+1)$ extending $T^{t}$. This exists
  by \zcref{lem:fund-ext}. For each $t \in 2^{\alpha}$ not in
  $\{t_{\alpha,0},\dots,t_{\alpha,n-1}\}$ define
  $T^{t^{\frown}\tup{0}}$ and $T^{t^{\frown}\tup{1}}$ both be equal
  to ${T'}^{t}$.  
  If any of the following conditions fail, do the
  same for all the $t \in \{t_{\alpha,0},\dots,t_{\alpha,n-1}\}$.
  \begin{itemize}
    \item For all $i < j < n$, $t_{\alpha,i} \neq t_{\alpha,j}$.
    \item $c_{\alpha}$ is unbounded in $\alpha$.
    \item $A_{\alpha} \seb \bigotimes_{i}T^{t_{\alpha,i}}$ is such that
    no two elements of $A_{\alpha}$ conflict on $c_{\alpha}$, and is maximal
    with respect to this property.
  \end{itemize}
  Otherwise, apply \zcref{lem:hard-ext}
  to $T^{t_{\alpha,0}},\dots,T^{t_{\alpha,n-1}}$, using $A = A_{\alpha}$
  and $C = c_{\alpha}$,
  to obtain $b_{0},\dots,
  b_{n-1}$. Then, for all $i < n$ define
  both $T^{t_{\alpha,i}^{\frown}\tup{0}}$ and
  $T^{t_{\alpha,i}^{\frown}\tup{1}}$ as the result of adjoining
  $b_{i}$ to $T^{t_{\alpha,i}}$. This finishes the construction
  of our family.

  Let $\T = \{T^{t} : t \in 2^{\omega_{1}}\}$. We claim
  that $\prop(\T)$ holds. For this fix $1 \le n < \omega$
  and $t_{0},\dots,t_{n-1}$ distinct elements of $2^{\omega_{1}}$.
  Also fix $C$ a club and $A \seb \bigotimes_{i < n}T^{t_{i}}$
  be such that no two elements of $A$ conflict on $C$.
  Extending $A$ if necessary, we may assume that $A$ is maximal
  with this property. We now show that $A$ is countable.

  Let $\gamma < \omega_{1}$ be such that for all $i < j < n$,
  $t_{i}\res \gamma \neq t_{j}\res \gamma$. One easily finds
  a club $D \seb C$ such that for all $\alpha \in D$,
  $\alpha$ is limit of elements in $C$ and $A \res \alpha$ is a maximal
  subset of $\bigotimes_{i < n}T^{t_{i}}\res \alpha$ having
  no two elements conflicting on $C \cap \alpha$. We may
  also assume that $\gamma < \min(D)$. Using the
  diamond property, let $\alpha \in D$ be such that
  $C \cap \alpha = c_{\alpha}$, $t_{i}\res \alpha = t_{\alpha,i}$
  for all $i < n$, and $A \res \alpha = A_{\alpha}$. Let
  $f : n \to 2$ be such that $t_{i}\res(\alpha+1) = t_{\alpha,i}^{\frown}
  \tup{f(i)}$.

  We claim that $A = A_{\alpha}$. If not, there is $a \in A$ of
  height at least $\alpha$. The maximality implies that in fact
  $A$ must have an element of height exactly $\alpha$,
  that is, an element $(b_{0}',\dots,b_{n-1}')$ where
  each $b_{i}'$ is an element of $T^{t_{\alpha,i}^{\frown}\tup{f(i)}}$
  of height $\alpha$. But 
  by the choice of $b_{0},\dots,b_{n-1}$ in
  the $\alpha$-th step of the construction, this implies that
  $(b_{0}',\dots,b_{n-1}')$ conflicts on $C$
  with an element of $A_{\alpha} \seb A$.
  So we conclude that $A = A_{\alpha}$.
\end{proof}

\section{Final remarks and open questions}

Is it possible to construct using $\dia$ (or stronger variations) a family
$\T$ of $2^{\omega_{1}}$-many suitable trees satisfying the stronger $\propp$?
A positive answer would render the work of \zcref{sec:non-dangerous}
unnecessary, since by \zcref{lem:psi+-safe}, the standard forcing
to obtain $\MA + \neg \CH$ would preserve $\prop(\T)$.

Observe that in our final model, we not only have $\prop(\T)$, but
$\propp(\T)$. So no ccc forcing can make the trees in $\T$
club-isomorphic again. Does $\propp$ have any other niceties over $\prop$?

Although we used $\dia$ to obtain the consistency of
``$2^{\omega_{1}} = \omega_{2}$
+ there is a family $\T$ of $2^{\omega_{1}}$-many suitable trees
satisfying $\prop(\T)$'', another option would be to iterate: The
starting point comes from the observation that given a suitable family satisfying
$\prop(\T)$, there is a $\T$-safe forcing extension in which there
is a suitable $U \notin \T$ such that $\prop(\T \cup \{U\})$. To see this,
start from a family $\T$ satisfying $\prop(\T)$,
we can first force with $\M(\T)$ so that it also holds $\propp(\T)$ and
that every
tree in $\T$ is special, and then add a single Cohen real, which
being a countable forcing, preserves both $\propp(\T)$ and that every
tree in $\T$ is special. In the resulting
model there is a suitable tree without stationary antichains
(see \cite{TodorcevicWalks}
and \cite{HrusakMartinez}). Now we use the following.

\begin{proposition}
  Assume $\propp(\T)$ and that every $T \in \T$ is special.
  Suppose that $U$ is
  a suitable tree with no stationary antichains, then $\prop(\T \cup \{U\})$.
\end{proposition}
\begin{proof}
  The result is trivial if $\T$ is empty. Fix $1 \le n < \omega$,
  distinct $T^{0},\dots,T^{n-1} \in \T$, a club $C$,
  and uncountable
  $A \seb T^{0}\otimes \cdots \otimes T^{n-1} \otimes U$. We need to
  show that there are $a,b \in A$ conflicting on $C$.
  By $\propp(\T)$ and the fact that $U$ is suitable, going to an uncountable
  subset we may assume that the elements of $\{(a(0),\dots,a(n-1)) : a \in A\}$
  are pairwise conflicting on $C$ and that $\{a(n) : a \in A\}$ is an antichain.

  Let $\{a_{\xi} : \xi < \omega_{1}\}$
  enumerate an uncountable subset of $A$ such that
  for each $\xi < \omega_{1}$, $\xi \le \hht(a_{\xi})$.
  For each $\xi < \omega_{1}$, let $b_{\xi} := (a_{\xi}(0)\res \xi,\dots,
  a_{\xi}(n)\res \xi)$. Because for each $i < n$, $T^{i}$ is special,
  we find stationary $\Gamma \seb C$ such that for all
  $\xi < \eta$ and for all $i < n$, $b_{\xi}(i)$ is incomparable
  with $b_{\eta}(i)$. Now because $U$ has no stationary antichains,
  we find $\xi < \eta$ in $\Gamma$ such that
  $b_{\xi}(n) \le_{T} b_{\eta}(n)$. Observe then
  that for all $i < n$,
  $\Delta(a_{\xi}(i),a_{\eta}(i)) < \xi \le \Delta(a_{\xi}(n),a_{\eta}(n))$.
  Since $\xi \in C$, we conclude that $a_{\xi}$ and $a_{\eta}$ conflict on $C$.
\end{proof}

Of course, to have the family have size $2^{\omega_{1}}$ one has to make
an iteration of length $\omega_{2}$ preserving $2^{\omega_{1}} = \omega_{2}$,
which requires more care.

\bibliographystyle{amsrefs}

\end{document}